\documentclass{article}
\usepackage{latexsym,amsfonts,euscript,amsmath,graphicx}
\usepackage{amssymb}
\usepackage{amsthm}
\usepackage{stmaryrd}
\usepackage{mathrsfs}
\usepackage{leftidx}
\usepackage{indentfirst}
\usepackage{color}
\usepackage{mathabx}

\usepackage[colorlinks,
linkcolor=blue,
anchorcolor=blue,
citecolor=blue]{hyperref} 

\makeatletter 
\@addtoreset{equation}{section}
\makeatother  

\def\la{\langle}\def\ra{\rangle}

\def\pfend{\hfill{$\Box$}}

\newcommand{\GL}{{\operatorname{GL}}}

\newcommand{\SL}{{\operatorname{SL}}}
\newcommand{\PSL}{{\operatorname{PSL}}}

\def\irr#1{{\rm Irr}(#1)}
\def\cent#1#2{{\bf C}_{#1}(#2)}

\def\z#1{{\bf Z}(#1)}

\def\M#1{\mathrm{M}(#1)}

\def\o#1{\overline{#1}}

\def\cod#1{\operatorname{cod}(#1)}
\def\cd#1{\operatorname{cd}(#1)}

\newtheorem{lem}{ \bf Lemma}[section]

\newtheorem{prop}[lem]{\bf Proposition}
\newtheorem{thm}[lem]{\bf Theorem}
\newtheorem*{thm*}{\bf Theorem}
\newtheorem*{con*}{\bf Conjecture}
\newtheorem*{thmA}{\bf Theorem~A}
\newtheorem*{corB}{\bf Corollary~B}

\title{Nonsolvable groups with three nonlinear irreducible character codegrees  
\thanks{{\bf Acknowledgement:} The first author gratefully acknowledge the support of the NSF of China (No. 11971391, 12071376). The authors also would like to thank Prof. Guohua Qian for many useful conversations on this topic.}}

 \author{Dongfang Yang\footnote{Email: dfyang1228@163.com}\qquad Yu Zeng\footnote{Corresponding author. Email: yuzeng2004@163.com}\\
  {\footnotesize\small  Dept. Mathematics, Changshu Institute of
  Technology, Changshu, Jiangsu, 215500, China}}

\date{}

\begin{document}
\maketitle

\vskip 1cm

\begin{center}\textbf{Abstract}\end{center}
For an irreducible character $\chi$ of a finite group $G$, the \emph{codegree} of $\chi$ is defined as $|G:\ker(\chi)|/\chi(1)$. 
In this paper, we determine finite nonsolvable groups with exactly three nonlinear irreducible character codegrees, and they are $\mathrm{L}_2(2^f)$ for $f\ge 2$, $\mathrm{PGL}_2(q)$ for odd $q\ge 5$ or $\mathrm{M}_{10}$.

\vskip 5cm

\bigskip


\textbf{2020 MR Subject Classification}\,\, 20C15, 20D25.
\pagebreak

\section{Introduction}


In the last decade, there was considerable interest in studying the character codegrees of finite groups.
For instance, there are interesting connections among character codegrees and element
orders (see \cite{isaacs2011,qian2021}).

Recall that if $G$ is a finite group and $\chi$ is an irreducible complex character of $G$, the \emph{codegree} of $\chi$ is defined (in \cite{qianwangwei2007}) to be
$${\rm cod}(\chi) =\frac{|G: \ker(\chi)|}{\chi(1)},$$
and we write $\cod{G}=\{ \cod{\chi} : \chi \in \irr{G} \}$ to denote the \emph{set of irreducible complex character codegrees} of $G$.
As usual, we use the standard notation $\cd{G}$ to denote the \emph{set of irreducible complex character degrees} of $G$.
Similar to the study on finite groups with few character degrees, 
the study on finite groups with few character codegrees appears extensively in the literature.
It has taken its first step by Alizadeh and his collaborators (see, for instance, \cite{alizadeh2019}): they gave a characterization of finite groups $G$ with $|\cod{G}|\le 3$.
Liu and Yang classified nonsolvable groups with exactly four (resp. five) character codegrees in \cite{liuyang2021four} (resp. \cite{liuyang2021five}).


Let $G$ be a finite group.
Given that $N \unlhd G$, we write
\[
	{\rm Irr}(G|N)=\{ \chi\in \irr{G}: \ker(\chi)\nleq  N\}.
\]
We also write $\cd{G|N}$ and $\cod{G|N}$ to denote the sets of degrees and codegrees,
respectively, of the characters in $\irr{G|N}$.
In particular, ${\rm Irr}(G|G')$ is exactly the set of nonlinear
 irreducible characters of $G$. 
 For a group $G$ with $|{\rm cd}(G|G')|=1$ (i.e. all nonlinear
 irreducible characters of $G$ have the same degree), Isaacs
 \cite[Corollary 12.6]{isaacs1994} showed that the derived subgroup $G'$
 is abelian.
 Analogously, Qian and the second author classified the finite nonnilpotent groups $G$ with $|{\rm cod}(G|G')|=1$ in \cite{qianzeng2022}.
 It is easy to show by Theorem \ref{cd<=2} that finite groups $G$ with $|\cod{G|G'}|\le 2$ are solvable.
 So, in this paper, we determine nonsolvable groups such that $|\cod{G|G'}|=3$.

\begin{thmA}
   Let $G$ be a finite nonsolvable group such that $|\cod{G|G'}|=3$.
   Then $G$ is either $\mathrm{L}_2(2^f)$ where $f\ge 2$, $\mathrm{PGL}_2(q)$ where $q=p^f$ is an odd prime power larger than $3$, or $\mathrm{M}_{10}$.
\end{thmA}

Obviously, $|{\rm cd}(G)|-|{\rm cd}(G|G')|=1$. However, $|{\rm cod}(G)|-|{\rm cod}(G|G')|$ is quite different:
 it can be arbitrarily large;
 it equals 1 only for perfect groups;
 it has lower bound $2$ for non-perfect groups (see Lemma \ref{p203}).
So, the next result (\cite[Theorem]{liuyang2021four}) is a direct consequence of Theorem A.

\begin{corB}
	If $G$ is a finite nonsolvable group such that $|\cod{G}|=4$,
	then $G$ is isomorphic to $\mathrm{L}_2(2^f)$ for $f\ge 2$.
\end{corB}

In the following, all groups considered are finite and $p$ always denotes a prime.
We use standard notation in character theory, as in \cite{isaacs1994}.

\section{Proofs}



We first present a lemma revealing the cardinality difference of $\cod{G}$ and $\cod{G|G'}$.

\begin{lem}[\cite{qianzeng2022}]\label{p203} Let $G$ be a group such that $p$ is the smallest prime divisor of $|G/G'|$.
	Then $p\not\in {\rm cod}(G|G')$. 
	In particular, $|{\rm cod}(G)|-|{\rm cod}(G|G')|\geq 2$. 
	\end{lem}
	


    The next proposition is the key to prove Theorem A.
	Let $G$ be a group and $H$ its subgroup, and let $\theta\in \irr{H}$.
	To state the proof of the next result, we shall use the notation $\irr{G|\theta}$ to denote the set of irreducible characters of $G$ lying over $\theta$; $\cod{G|\theta}$ to denote the set of codegrees of the characters in $\irr{G|\theta}$;
	$\mathrm{M}(G)$ to denote the Schur multiplier of a group $G$;
	$n_p$ to denote the largest power of $p$ that divides the positive integer $n$; 
	 $(C_p)^n$ to denote the elementary abelian $p$-group of order $p^n$ where $n$ is a positive integer.

\begin{prop}\label{pro}
	Let $N$ be a nontrivial normal subgroup of a group $G$.
	If $G/N$ is an almost simple group with socle $\mathrm{L}_{2}(q)$ where $q=p^f$ for some prime $p$, then $\cod{G/N|(G/N)'}\subsetneq \cod{G|G'}$.
\end{prop}
\begin{proof}
	Suppose false and let $G$ be a counterexample of minimal possible order.
    Then $N$ is minimal normal in $G$ and $\cod{G/N|(G/N)'}=\cod{G|G'}$.
	Notice that $G/N$ is an almost simple group with socle $K/N\cong\mathrm{L}_{2}(q)$,
	and hence an application of \cite[Theorem A]{white2013} yields
	\[
		\cd{G/N} \subseteq \{ 1,q,\frac{q+\varepsilon}{2} \} \cup \{ (q-1)j: j\mid b \}\cup \{ (q+1)j : j\mid b \}
	\]
	where $\varepsilon=(-1)^{(q-1)/2}$ and $b=|G/N:H/N|$.
	Furthermore,
	if $G/N$ contains the diagonal automorphism $\mathfrak{d}$ of order 2, then $H/N=K/N \rtimes \la \mathfrak{d}\ra$, $q$ is odd and 
	\begin{equation}\label{e_1}
		\cod{G/N|(G/N)'} \subseteq \{ b(q^{2}-1)  \}\cup \{ \frac{bq(q+1)}{j}: j\mid b \}\cup \{ \frac{bq(q-1)}{j} : j\mid b \}; 
	\end{equation}
	if $G/N$ does not contain the diagonal automorphism of order 2, then $H/N=K/N$ and
	\begin{equation}\label{e_2}
		\cod{G/N|(G/N)'} \subseteq \{ \frac{b(q^{2}-1)}{(2,q-1)}, \frac{2bq(q^2-1)}{(2,q-1)(q+\varepsilon )}  \}\cup \{ \frac{bq(q+1)}{(2,q-1)j}: j\mid b \}\cup \{ \frac{bq(q-1)}{(2,q-1)j} : j\mid b \}.
	\end{equation}
	Now, we will complete the proof by carrying out the following steps.

	\emph{Step 1. $N$ is the unique minimal normal subgroup of $G$.}
	
    Otherwise, $S$ is the other minimal normal subgroup of $G$ where $S\cong \mathrm{L}_2(q)$.
    Let $\varphi$ be the Steinberg character of $S$.
	
	Assume first that $N$ is solvable.
	If $G/S$ is nonabelian, then $G$ has a quotient group which is Frobenius. 
	So, $k\in \cod{G|G'}$ where $k\mid 2b$.
    However, since $|K/N|_r>|G/K|_r$ for every $r\in \pi(K/N)$ by \cite[Proposition 2.5]{qian2012 p-closed}, we conclude a contradiction $k \notin \cod{G/N|(G/N)'}$ by calculation.
	In fact, since $q$, $q+1$ and $(q-1)/(2,q-1)$ are Hall numbers of $|K/N|$ which are relatively coprime,
	$k\in \cod{G/N|(G/N)'}$ implies that there exists a prime $r\in \pi(K/N)$ such that $|K/N|_r\le |G/K|_r$. 
	
	Hence, $G/S$ is abelian, that is $N \cong C_{\ell}$ (for a prime $\ell$) is central in $G$.
    Let $\theta \in \irr{N|N}$ and let $\psi=\varphi \times \theta$,
	and observe that $\psi \in \irr{S\times N}$ is $G$-invariant.
	If $G/N$ does not contain the diagonal automorphism of order $2$, then $G/K$ is cyclic,
	and hence $\psi$ extends to a faithful $\chi \in \irr{G}$.
	So, 
	$$\cod{\chi}=\frac{|G|}{\chi(1)}=\frac{|G/N|}{\varphi(1) }\frac{|N|}{\theta(1)}=\frac{(q^2-1)b\ell}{(2,q-1)}.$$
	By checking (\ref{e_2}), $\cod{\chi}\notin \cod{G/N|(G/N)'}$ as $q$, $q+1$ and $(q-1)/(2,q-1)$ being relatively coprime.
    If $G/N$ contains the diagonal automorphism of order $2$, then $q$ is odd.
	Observe that $S$ has exactly two faithful irreducible characters $\alpha_1$ and $\alpha_2$ with degree $(q+\varepsilon)/2$ which are conjugate under $H$ (see, for instance, \cite[Lemma 4.5]{white2013}), and then $(\alpha_1\times \theta)^H=(\alpha_2\times \theta)^H \in \irr{H}$ has degree $q+\varepsilon$.
	Since $H\unlhd G$ such that $G/H$ is cyclic,  
    and hence $(\alpha_1\times \theta)^H$ extends to a faithful $\omega \in \irr{G}$.
	So, $\cod{\omega}=|G|/\omega(1)=b\ell q(q-\varepsilon)$ is not in $\cod{G/N|(G/N)'}$ by checking (\ref{e_1}), a contradiction.



	Hence, $N$ is nonsolvable. 
   Then by \cite[Theorem 2, Theorem 3, Theorem 4, Lemma 5]{bianchi2007} there exists an irreducible character $\theta$ of $N$ which is extendible to $G$.
   By the similar reasoning in the above paragraph, we deduce by calculation that $\cod{G|\psi}\cap \cod{G/N|(G/N)'}=\varnothing$, a contradiction.


	\emph{Step 2. $N$ is solvable.}

	Otherwise, $N$ is nonsolvable. By \cite[Theorem 2, Theorem 3, Theorem 4, Lemma 5]{bianchi2007} there exists a nonlinear irreducible character $\theta$ of $N$ which extends to $\chi \in \irr{G}$. 
	Observe that by Gallagher correspondence $\irr{G|\theta}=\{ \chi\beta: \beta\in \irr{G/N} \}$ and that $\ker(\chi\beta)=1$ for all $\beta \in \irr{G/N}$, one has $\cod{\chi\beta}=\frac{|G|}{\theta(1)\beta(1)}$.
    In particular, for each faithful $\alpha \in \irr{G/N}$, $\cod{\chi\alpha}=\cod{\alpha}\frac{|N|}{\theta(1)}$ is not in $\cod{G/N|(G/N)'}$, a contradiction.

    \emph{Step 3. $\cent{G}{N}=N$.}
	
	Otherwise, $K \le \cent{G}{N}$.
	By the uniqueness of $N$, $N\le \z K \cap K'$, and hence $N$ embeds into $\M{K/N}\cong \M{\mathrm{L}_2(q)}$.
	Since $N$ is an elementary abelian $\ell$-group, either $|N|=2$ or $q=9$ and $|N|=3$.
	Assume the latter. 
	Then $K\cong 3.A_6$ and $G/K$ is a subgroup of $(C_2)^2$.
    Let $\varphi \in \irr{K|N}$ be the character of degree $9$, and observe that $\det (\varphi)=1_K$ as $K$ being quasisimple.
    An application of \cite[Corollary 6.28]{isaacs1994} yields that $\varphi$ extends to $\chi \in \irr{G}$.	
    By calculation, we deduce a contradiction $\cod{\chi}=|G|/\chi(1) \notin \cod{G/N|(G/N)'}$.
	
	So, $K\cong \SL_2(q)$ where $q$ is odd and $|N|=2$.
	It can be checked that there exists a faithful $\alpha \in \irr{K}$ of degree $q+1$ which is $G$-invariant.
	If $G/N$ does not contain the diagonal automorphism of order 2, then $G/K$ is cyclic.
	Therefore, 
	$\alpha$ extends to $\chi\in \irr{G}$,
	and hence $\cod{\chi}=2q(q-1)b$ which is not in $\cod{G/N|(G/N)'}$, a contradiction.
    Suppose that $G/N$ contains the diagonal automorphism of order 2.
    Then it can be checked that there exists a faithful $\omega \in \irr{G}$ of degree $q-\varepsilon$.
    In fact, $K$ has exactly two faithful irreducible characters $\beta_1$ and $\beta_2$ with degree $(q-\varepsilon)/2$ which are conjugate under $H$;
	as $|H:K|=2$, $\beta_1^H \in \irr{H}$ has degree $q-\varepsilon$;
	note that $H\unlhd G$ such that $G/H$ is cyclic and that $\irr{H|\beta_1}=\irr{H|\beta_2}=\{ \beta_1^H \}$,  
    and hence $\beta_1^H$ is extendible to $G$.
	So, $\cod{\omega}$ is not in $\cod{G/N|(G/N)'}$, a contradiction.

	\emph{Step 4. Final contradiction.}

    Write $|N|=\ell^n$ for some prime $\ell$.
	For each $\chi \in \irr{G|N}$, $\ker(\chi)=1$ and $\ell^n \mid \chi(1)$.
	Since
	\[
		\ell^{2n}\mid\Sigma_{\chi\in \irr{G|N}} \chi(1)^2=|G|-|G/N|,
	\] 
	$\ell^{2n}\mid 2_\ell b_{\ell} |\mathrm{L}_2(q)|_{\ell}$,
	and hence by \cite[Proposition 2.5]{qian2012 p-closed} $\ell^{2n}\mid  (|\mathrm{L}_2(q)|_{\ell})^2$.
	So, $\ell^{2n} < (q+1)^2$, i.e. $\ell^{n}< q+1$.
    As $\cent{G}{N}=N\cong (C_{\ell})^n$, $G/N$ embeds into $\GL_{n}(\ell)$.
    Recall that $G/N$ is nonsolvable, and hence $\ell^{n}\geq 8$, i.e. $q>7$.
    Set $\o G=G/N$, let $\o F$ be a Frobenius subgroup of $\o K$ of order $q(q-1)/(2,q-1)$, and let $\o E$ be the Frobenius kernel of $\o F$.
	Set $N_0=\cent{N}{\o E}$, then $V=N/N_0$ is an $\o F$-module such that $\cent{V}{\o E}=1$.
    If $\ell \neq p$, i.e. $(|V|,|\o E|)=1$, then $n\ge \dim (V)\ge (q-1)/(2,q-1)$ by \cite[Theorem 15.16]{isaacs1994}.
	Therefore, 
	$$q+1> \ell^n\geq \ell^{(q-1)/(2,q-1)},$$
    a contradiction.
	In fact, $q+1\le 2^{(q-1)/2}$ where $q>7$.
    So, $\ell=p$ and $n\le f$.	
	Assume that $f\geq 2$.
	By Zsigmondy's theorem, there exists a prime divisor of $p^f+1$ which does not divide $|\GL_n(p)|$ except for the case $p^f=2^3$.
	If $p^f=2^3$, then $p^f+1 \nmid |\GL_n(p)|$ contradicting the fact that $G/N$ embeds into $\GL_n(p)$.
	So, $n=f=1$, and hence $G/N$ is solvable, the final contradiction.
\end{proof}

The following deep theorem \cite[Theorem B]{isaacsknutson1998} is another fact needed to prove Theorem A which is also the key to show that a group $G$ with $|\cod{G|G'}|\le 2$ is solvable.

\begin{thm}\label{cd<=2}
   Let $N$ be a normal subgroup of a group $G$ and suppose that $|\cd{G|N}|\leq 2$, then $G$ is solvable.
\end{thm}

Now, we are ready to prove Theorem A.

\bigskip

\noindent \emph{Proof of Theorem A.}
    Assume false and let $G$ be a counterexample of minimal possible order.
	Let $N\unlhd G$ be maximal such that $G/N$ is nonsolvable.
	Then $G/N$ has the unique minimal normal subgroup, say $K/N$, which is nonabelian.
	
	Assume that $N=1$. 
	Recall that $K$ is the unique minimal normal subgroup of $G$ which is nonabelian.
	For a nonprincipal irreducible character $\theta$ of $K$ and an irreducible character $\chi$ of $G$ lying above $\theta$, $\ker(\chi)=1$ and hence $\cod{\chi}=|G|/\chi(1)$.
    So, $|\cd{G|K}|\le |\cod{G|K}|\le |\cod{G|G'}|\le 3$ where the second inequality holds because $K\le G'$.	 
    As $G$ is nonsolvable, it follows by Theorem \ref{cd<=2} that $|\cd{G|K}|=3$.
	Therefore, \cite[Theorem 1.1 and Theorem 1.2]{heqian2012} yields a contradiction.

	So, $N>1$, and write $\o G=G/N$. 
	By the discussion in the above paragraph, one see that $|\cod{\o G|{\o G}'}|=3$ as $\o G$ being nonsolvable.
	By the minimality of $G$, $\o G$ is either $\mathrm{L}_2(2^f)$ where $f\ge 2$, $\mathrm{PGL}_2(q)$ where $q=p^f$ is an odd prime power larger than 3, or $\mathrm{M}_{10}\cong \mathrm{L}_2(q) \rtimes C_2$. 
	So, we conclude from Proposition \ref{pro} that $|\cod{G|G'}|>|\cod{\o G|{\o G}'}|=3$, the final contradiction. \pfend
    
\bigskip

Following similar argument in the above proof, one could prove that a group $G$ with $|\cod{G|G'}|\le 2$ is solvable.  
It also explains why we study nonsolvable groups with exactly three nonlinear irreducible character codegrees.

\end{document}